\documentclass[twoside,english,11pt]{amsart}
\usepackage[T1]{fontenc}
\usepackage[latin9]{inputenc}
\usepackage{mathrsfs}
\usepackage{amstext}
\usepackage{amsthm}
\usepackage{amssymb}
\usepackage{fancyhdr}
\makeatletter
\numberwithin{equation}{section}
\numberwithin{figure}{section}
\theoremstyle{plain}
\newtheorem{thm}{\protect\theoremname}
\theoremstyle{remark}
\newtheorem{rem}[thm]{\protect\remarkname}
\theoremstyle{definition}
\newtheorem{defn}[thm]{\protect\definitionname}
\theoremstyle{plain}
\newtheorem{lem}[thm]{\protect\lemmaname}
\theoremstyle{remark}
\newtheorem*{rem*}{\protect\remarkname}
\theoremstyle{remark}
\newtheorem{claim}[thm]{\protect\claimname}
\theoremstyle{remark}
\newtheorem*{acknowledgement*}{\protect\acknowledgementname}

\makeatother

\usepackage{babel}
\providecommand{\acknowledgementname}{Acknowledgement}
\providecommand{\claimname}{Claim}
\providecommand{\definitionname}{Definition}
\providecommand{\lemmaname}{Lemma}
\providecommand{\remarkname}{Remark}
\providecommand{\theoremname}{Theorem}

\begin{document}

\title{Interior curvature estimates for hypersurfaces of prescribing scalar
curvature in dimension three}
\author{GUOHUAN QIU }
\begin{abstract}
We prove a priori interior curvature estimates for hypersurfaces of
prescribing scalar curvature equations in $\mathbb{R}^{3}$. The method
is motivated by the integral method of Warren and Yuan in \cite{warren2009hessian}.
The new observation here is that the ``Lagrangian'' submanifold
constructed similarly as Harvey and Lawson in \cite{harvey1982calibrated}
has bounded mean curvature if the graph function of a hypersurface
satisfies the scalar curvature equation.
\end{abstract}

\thanks{This work is supported by Direct Grant, CUHK}
\address{Department of Mathematics, The Chinese University of Hong Kong, Shatin,
N.T., Hong Kong, China. }
\email{ghqiu@math.cuhk.edu.hk}
\maketitle

\section{Introduction}

In this paper, we are studying the regularity thoery of a hypersurface
$M^{n}\subseteq\mathbb{R}^{n+1}$ with positive scalar curvature $R_{g}>0$.
In hypersurface geometry, the Gauss equation tells us
\[
R_{g}=\sigma_{2}(\kappa):=\sum_{1\leq i_{1}<i_{2}\leq n}\lambda_{i_{1}}\lambda_{i_{2}}
\]
 where $\kappa(x)=(\lambda_{1}(x),\cdots,\lambda_{n}(x))$ are the
principal curvatures of the hypersurface.

Suppose $M^{n}$ is a $C^{1}$ graph $X=(x,u(x))$ over $x\in B_{r}\subseteq\mathbb{R}^{n}$.
In this setting, the scalar curvature equation which we study in this
paper is
\begin{equation}
\sigma_{2}(\kappa(x))=f(X(x),\nu(x))>0,\label{eq:scalar}
\end{equation}

where $\nu$ is the normal of the given hypersurface as a graph over
a ball $B_{r}\subset\mathbb{R}^{n}$. This is the second order elliptic
PDE depends on the graph function $u$. If $n=2$, it is Monge-Ampere
equation
\begin{equation}
\det(u_{ij})=f(x,u,\nabla u).\label{eq:MA}
\end{equation}

Our study of the scalar curvature equation is motived by isometric
embedding problems. A famous isometric embedding problem is the Weyl
problem. It is the problem of realizing, in three-dimensional Euclidean
space, a regular metric of positive curvature given on a sphere. This
is the Weyl problem which was finally solved by Nirenberg \cite{nirenberg1953weyl}
and Pogorelov \cite{pogorelov1973extrinsic} independently. They solved
the problem of Weyl by a continuity method where obtaining $C^{2}$
estimate to the scalar curvature equation is important to the method.

Also motived by the Weyl problem, E. Heinz \cite{heinz1959elliptic}
first derived a purely interior estimate for the equation (\ref{eq:MA})
in dimension two. If $u$ satisfies the equation (\ref{eq:MA}) in
$B_{r}\subseteq\mathbb{R}^{2}$ with positive $f$, then
\begin{equation}
\sup_{B_{\text{\ensuremath{\frac{r}{2}}}}}\lvert D^{2}u\rvert\leq C(\lvert u\vert_{C^{1}(B_{r})},\lvert f\vert_{C^{2}(B_{r})},\inf_{B_{r}}f).\label{eq:interior}
\end{equation}
And this type of estimate turns out to be very useful when one study
the isometric embedding problem for surfaces with boundary or for
non-compact surfaces. But Heinz's interior $C^{2}$ estimate is false
for the convex solutions to the equation $\det D^{2}u=1$ $in$ $B_{r}\subseteq\mathbb{R}^{n}$
when $n\geq3$ by Pogorelov \cite{pogorelov1978multidimensional}
.

The second motivation is from the studying of fully nonlinear partial
differential equation theory itself. Caffarelli-Nirenberg-Spruck started
to study $\sigma_{k}-$Hessian operators and established existence
of Dirichlet problem for $\sigma_{k}$ equations in their seminal
work \cite{caffarelli1985dirichlet}. Here the $\sigma_{k}-$Hessian
operators are the $k-$th elementary symmetric function for $1\le k\leq n$.
The key to the existence of Dirichlet problem is by establishing $C^{2}$
estimates up to the boundary.
\[
\sup_{\bar{\Omega}}\lvert D^{2}u\rvert\leq C(\lvert u\rvert_{C^{1}(\bar{\Omega})},f,\varphi,\partial\Omega).
\]
Although there are $C^{2}$ estimates to $\sigma_{k}-$Hessian equations
for boundary value problems, there are, in general, no interior $C^{2}$
estimates to $\sigma_{k}-$Hessian equations. Because the Pogorelov's
counter-examples were extended in \cite{urbas1990existence} to $k\geq3$.
The best we can expect is the Pogorelov type interior $C^{2}$ estimates
with homogeneous boundary data which were derived by Pogorelov \cite{pogorelov1978multidimensional}
for $k=n$ and by Chou-Wang for $k<n$ \cite{chou2001variational}
. So people in this field want to know whether the interior $C^{2}$
estimate for $\sigma_{2}$ equations holds or not for $n\geq3$. In
fact, the interior regularity for solutions of the following $\sigma_{2}$-Hessian
equation and prescribing scalar curvature equation is a longstanding
problem,
\begin{equation}
\sigma_{2}(\nabla^{2}u)=f(x,u,\nabla u)>0,\begin{array}{ccc}
 & x\in B_{r}\subset\mathbb{R}^{n}.\end{array}\label{eq:sigma2}
\end{equation}
and
\[
\sigma_{2}(\kappa(x))=f(X(x),\nu(x))>0.
\]

A major breakthrough was made by Warren-Yuan \cite{warren2009hessian}.
They obtained $C^{2}$ interior estimate for the equation
\begin{equation}
\sigma_{2}(\nabla^{2}u)=1.\begin{array}{ccc}
 & x\in B_{1}\subset\mathbb{R}^{3}\end{array}\label{eq:specialLag}
\end{equation}
Recently in \cite{MSY}, McGonagle-Song-Yuan proved interior $C^{2}$
estimate for convex solutions of the equation $\sigma_{2}(\nabla^{2}u)=1$
in any dimensions. Using a different argument, the interior $C^{2}$
estimates for solutions of more general equations (\ref{eq:sigma2})
and (\ref{eq:scalar}) in any dimensions with certain convexity constraints
were also obtained by Guan-Qiu in \cite{guan2017interior}. Moreover,
we proved interior curvature estimates for isometrically immersed
hypersurfaces in $\mathbb{R}^{n+1}$ with positive scalar curvature
in that paper \cite{guan2017interior}.

In this paper, we completely solve this problem for scalar equations
in dimension three.
\begin{thm}
\label{thm:scalar} Suppose $M$ is a smooth graph over $B_{10}\subset\mathbb{R}^{3}$
with positive scalar curvature and it is a solution of equation (\ref{eq:scalar}).
Then we have
\begin{equation}
\sup_{x\in B_{\frac{1}{2}}}|\kappa(x)|\leq C,
\end{equation}
where $C$ depends only on $||M||_{C^{1}(B_{10})}$ $\|f\|_{C^{2}(B_{10}\times\mathbb{S}^{2})}$,
and $\|\frac{1}{f}\|_{L^{\infty}(B_{10}\times\mathbb{S}^{2})}$ .
\end{thm}

Analogously we have
\begin{thm}
\label{sigma2Equation} Let $u$ be a solution to (\ref{eq:sigma2})
on $B_{10}\subset\mathbb{R}^{3}$ . Then we have
\begin{equation}
\sup_{B_{\frac{1}{2}}}|D^{2}u|\leq C,
\end{equation}
where $C$ depends only on $\|f\|_{C^{2}(B_{10}\times\mathbb{R}\times\mathbb{R}^{3})}$,
$\|\frac{1}{f}\|_{L^{\infty}(B_{10}\times\mathbb{R}\times\mathbb{R}^{3})}$
and $||u||_{C^{1}(B_{10})}$.
\end{thm}

\begin{rem}
The proof of Theorem \ref{sigma2Equation} is similar to Theorem \ref{thm:scalar}.
So we will omit its proof which can be found in a recent paper \cite{qiu2017interior}
by the author. The method given here is easier than to the previous
one, so we will not submit the previous paper \cite{qiu2017interior}.
\end{rem}

In order to introduce our idea, let us briefly review the ideas for
attacking this problem so far. In two dimensional case, Heinz used
Uniformization theorem to transform this interior estimate for Monge-Ampere
equation into the regularity of an elliptic system and univalent of
this mapping, see also \cite{heinz1956certain,lu2016weyl} for more
details. Another interesting proof using only maximum principle was
given by Chen-Han-Ou in \cite{chen2016interior}. Our new quantity
in \cite{guan2017interior} can give a new proof of Heinz. The restriction
for these methods is that we need some convexity conditions which
are not easily got in the higher dimension.

In $\mathbb{R}^{3}$, a key observation made in \cite{warren2009hessian}
is that equation (\ref{eq:specialLag}) is exactly the special Lagrangian
equation which stems from the special Lagrangian geometry \cite{harvey1982calibrated}.
And an important property for the special Lagrangian equation is that
the Lagrangian graph $(x,Du)\subset\mathbb{R}^{3}\times\mathbb{R}^{3}$
is a minimal submanifold which has mean value inequality and sobolev
inequality. So Warren-Yuan have proved interior $C^{1}$ estimate
for the special Lagrangian submanifold which in turn proved interior
$C^{2}$ estimate for the special Lagrangian equation. Our new observation
in this paper is that the graph $(X,\nu)$, where $X$ is position
vector of the hypersurface whose scalar curvature satisfy equation
(\ref{eq:sigma2}), can be viewed as a submanifold in $\mathbb{R}^{4}\times\mathbb{R}^{4}$
with bounded mean curvature. Then applying similar arguement of Michael-Simon
\cite{michael1973sobolev}, see also Hoffman-Spruck \cite{hoffman1974sobolev},
we have a mean value inequality in order to remove the convexity condition
in \cite{guan2017interior}. Finally, we apply a modified argument
of Warren-Yuan in \cite{warren2009hessian} to get the estimate.

At last, we remark that the arguments are higher co-dimensional analogous
to the original integral proof by Bombieri-De Giorgi-Miranda \cite{bombieri1969maggiorazione}
for the gradient estimate for co-dimension one minimal graph and by
Ladyzhenskaya and Ural'Tseva \cite{ladyzhenskaya1970local} for general
prescribed mean curvature equations. Here we use the method very similar
to Trudinger's simplified proof of gradient estimate for mean curvature
equations in \cite{trudinger1972new,trudinger1973gradient}.

The higher dimensional cases for these equations are still open to
us.

\section{Preliminary Lemmas}

We first introduce some definitions and notations.
\begin{defn}
\label{defsigma}For $\lambda=(\lambda_{1},\cdots,\lambda_{n})\in\mathbb{R}^{n}$,
the $k$- th elementary symmetric function $\sigma_{k}(\lambda)$
is defined as
\[
\sigma_{k}(\lambda):=\sum\lambda_{i_{1}}\cdots\lambda_{i_{k}},
\]
 where the sum is taken over for all increasing sequences $i_{1},\cdots,i_{k}$
of the indices chosen from the set $\{1,\cdots,n\}$. The definition
can be extended to symmetric matrices where $\lambda=(\lambda_{1},\cdots,\lambda_{n})$
are the corresponding eigenvalues of the symmetric matrices.
\end{defn}

For example, in $\mathbb{R}^{3}$
\[
\sigma_{2}(D^{2}u):=\sigma_{2}(\lambda(D^{2}u))=\lambda_{1}\lambda_{2}+\lambda_{1}\lambda_{3}+\lambda_{2}\lambda_{3}.
\]

\begin{defn}
For $1\leq k\leq n$, let $\Gamma_{k}$ be a cone in $\mathbb{R}^{n}$
determined by
\[
\Gamma_{k}=\{\lambda\in\mathbb{R}^{n}:\sigma_{1}(\lambda)>0,\cdots,\sigma_{k}(\lambda)>0\}.
\]
\end{defn}

The following lemma is from \cite{lin1994some}.
\begin{lem}
\label{lemma2} Suppose $\lambda\in\Gamma_{2}$ and $\sigma_{2}^{ii}(\lambda):=\frac{\partial\sigma_{2}}{\partial\lambda_{i}}$,
then there is a constant $c>0$ depending only on $n$ such that for
any $i$ from $1$ to $n$,
\[
\sigma_{2}^{ii}(\lambda)\geq\frac{c\sigma_{2}(\lambda)}{\sigma_{1}(\lambda)}.
\]
If we assume that $\lambda_{1}\ge\cdots\ge\lambda_{n}$, then there
exist $c_{1}>0$ and $c_{2}>0$ depending only on $n$ such that
\begin{equation}
\sigma_{2}^{11}(\lambda)\lambda_{1}\geq c_{1}\sigma_{2}(\lambda),\label{tildec0}
\end{equation}
and for any $j\geq2$
\begin{equation}
\sigma_{2}^{jj}(\lambda)\geq c_{2}\sigma_{1}(\lambda).\label{s222}
\end{equation}
\end{lem}

So the curvature estimates can be reduced to the estimate of mean
curvature $H$ due to the following fact
\begin{equation}
\max|\lambda_{i}|\leq H=\sigma_{1}(\kappa).\label{eq:lap}
\end{equation}
In the rest of this article, we will denote $C$ to be constant under
control (depending only on $\|f\|_{C^{2}}$, $\|\frac{1}{f}\|_{L^{\infty}}$
and $\|M\|_{C^{1}}$), which may change line by line.

Suppose that a hypersurface $M$ in $\mathbb{R}^{n+1}$ can be written
as a graph over $B_{r}\subseteq\mbox{\ensuremath{\mathbb{R}}}^{n}$.
At any point of $x\in B_{1}$, the principal curvature $\kappa=(\lambda_{1},\lambda_{2},\cdots,\lambda_{n})$
of the graph $M=(x,u(x))$ satisfy a equation
\begin{equation}
\sigma_{2}(\kappa)=f(X,\nu)>0,\label{eq:graph-1}
\end{equation}
where $X$ is the position vector of $M$, and $\nu$ a normal vector
on $M$.

We choose an orthonormal frame in $\mathbb{R}^{n+1}$ such that $\{e_{1},e_{2},\cdots,e_{n}\}$
are tangent to $M$. Let $\nu$ is a normal on $M$ such that $H>0$.
We recall the following fundamental formulas of a hypersurface in
$\mathbb{R}^{n+1}$:
\begin{eqnarray*}
X_{ij} & = & -h_{ij}\nu\begin{array}{cc}
 & (Gauss\,\,formula)\end{array}\\
\nu_{i} & = & h_{ij}e_{j}\begin{array}{cc}
 & (Weingarten\,\,equation)\end{array}\\
h_{ijk} & = & h_{ikj}\begin{array}{cc}
 & (Codazzi\,\,equation)\end{array}\\
R_{ijkl} & = & h_{ik}h_{jl}-h_{il}h_{jk}\begin{array}{cc}
 & (Gauss\,\,equation)\end{array},
\end{eqnarray*}
where $R_{ijkl}$ is the curvature tensor. We also have the following
commutator formula:
\begin{eqnarray}
h_{ijkl}-h_{ijlk} & = & h_{im}R_{mjkl}+h_{mj}R_{mikl}.\label{eq:commute}
\end{eqnarray}

Combining Codazzi equation, Gauss equation and (\ref{eq:commute}),
we have
\begin{eqnarray}
h_{iikk}=h_{kkii}+\sum_{m}(h_{im}h_{mi}h_{kk}-h_{mk}^{2}h_{ii}).\label{eq:commute2}
\end{eqnarray}

For scalar curvature equation (\ref{eq:scalar}) with positive scalar
curvature, we may assume that $M$ is admissible in the following
definition without loss of generality.
\begin{defn}
A $C^{2}$ surface $M$ is called admissible if at every point $X\in M$,
its principal curvature satisfies
\[
\kappa\in\Gamma_{2}.
\]
Moreover, for any symmetric matrix $h_{ij}$, it follows from Lemma
\ref{lemma2} that $\sigma_{2}^{ij}:=\frac{\partial\sigma_{2}(\lambda(h_{ij}))}{\partial h_{ij}}$
is positive definite if $\lambda(h_{ij})\in\Gamma_{2}$.
\end{defn}

\begin{lem}
\label{lem1} Suppose the scalar curvature of hypersurface $M$ satisfies
equation (\ref{eq:graph-1}). In orthonormal coordinate, we have the
following equations
\begin{equation}
\sigma_{2}^{kl}h_{kli}=\nabla f(e_{i}),\label{eq:f1}
\end{equation}
and
\begin{eqnarray}
\sigma_{2}^{kl}h_{iikl}+\sum_{k\neq l}h_{kki}h_{lli}-\sum_{k\neq l}h_{kli}h_{kli}\nonumber \\
-2f\sum_{k}h_{ki}^{2}+(f\sigma_{1}-3\sigma_{3})h_{ii} & = & \nabla^{2}f(e_{i},e_{i}).\label{eq:f2}
\end{eqnarray}
 If $f$ is a form with gradient term, then there are estimates

\begin{equation}
|\nabla f|\leq C(1+H),\label{eq:gradient}
\end{equation}
and
\begin{equation}
-C(1+H)^{2}+\sum_{k}h_{ij}^{k}d_{\nu}f(e_{k})\leq\nabla^{2}f(e_{i},e_{j})\leq C(1+H)^{2}+\sum_{k}h_{ij}^{k}d_{\nu}f(e_{k}).\label{eq:Hessianf}
\end{equation}
\end{lem}

\begin{proof}
Taking twice differential of the equation $\sigma_{2}(\kappa)=f$
, we get (\ref{eq:f1}) and
\[
\sigma_{2}^{kl}h_{klii}+\sum_{k\neq l}h_{kki}h_{lli}-\sum_{k\neq l}h_{kli}h_{kli}=\nabla^{2}f(e_{i},e_{i}).
\]

Then we obtain (\ref{eq:f2}) by (\ref{eq:commute2}) and the following
elementary identities
\[
\sigma_{2}^{ij}h_{ij}=2f,
\]
and
\[
\sum_{m}\sigma_{2}^{kl}h_{mk}h_{ml}=\sigma_{1}\sigma_{2}-3\sigma_{3}.
\]

Moreover, by (\ref{eq:lap}), Codazzi equation and the following direct
computations
\[
\nabla f(e_{i})=d_{X}f(e_{i})+h_{i}^{k}d_{\nu}f(e_{k}),
\]

and
\begin{eqnarray*}
\nabla^{2}f(e_{i},e_{j}) & = & d_{X}^{2}f(e_{i},e_{j})+h_{j}^{k}d_{X,\nu}^{2}f(e_{i},e_{k})-h_{ij}d_{X}f(\nu)+h_{i}^{k}d_{\nu,X}^{2}f(e_{k},e_{j})\\
 &  & +h_{i}^{k}h_{j}^{l}d_{\nu}^{2}f(e_{k},e_{l})-h_{i}^{k}h_{kj}d_{\nu}f(\nu)+h_{ij}^{k}d_{\nu}f(e_{k}),
\end{eqnarray*}

we get the estimates (\ref{eq:gradient}) and (\ref{eq:Hessianf}).
\end{proof}
We recall some elementary facts about hypersurface. Denoting $W=\sqrt{1+\lvert Du\rvert^{2}}$,
the second fundamental form and the first fundamental form of the
hypersurface can be written in local coordinate as $h_{ij}=\frac{u_{ij}}{W}$
and $g_{ij}=\delta_{ij}+u_{i}u_{j}$. The inverse of the first fundamental
form and the Weingarten Curvature are $g^{ij}=\delta_{ij}-\frac{u_{i}u_{j}}{W^{2}}$
and $h_{i}^{j}=D_{i}(\frac{u_{j}}{W})$.
\begin{defn}
The Newton transformation tensor is defined as
\[
[T_{k}]_{i}^{j}:=\frac{1}{k!}\delta_{jj_{1}\cdots j_{k}}^{ii_{1}\cdots i_{k}}h_{j_{1}}^{i_{1}}\cdots h_{j_{k}}^{i_{k}},
\]
and the corresponding $(2,0)$-tensor is defined as
\[
[T_{k}]^{ij}:=[T_{k}]_{k}^{i}g^{kj}.
\]
\end{defn}

From this definition one can easily show a divergence free identity
\[
\sum_{j}\partial_{j}[T_{k}]_{i}^{j}=0.
\]

\begin{lem}
There is a family of elementary relations between $\sigma_{k}$ opertators
and Newton transformation tensors
\end{lem}

\begin{eqnarray}
[T_{k}]_{i}^{j} & = & \sigma_{k}\delta_{i}^{j}-[T_{k-1}]_{i}^{l}h_{l}^{j}.\label{eq:Tk1}
\end{eqnarray}

or
\begin{equation}
[T_{k}]_{i}^{j}=\sigma_{k}\delta_{i}^{j}-[T_{k-1}]_{l}^{j}h_{i}^{l}.\label{eq:Tk2}
\end{equation}

Moreover, the $(2,0)$-tensor of $T_{k}$ is symmetry such that
\begin{equation}
[T_{k}]^{ij}=[T_{k}]^{ji}.\label{eq:Tij}
\end{equation}

\begin{proof}
We only prove the first one, because the second one is similar. From
Definition \ref{defsigma}, it is easy to check that
\begin{equation}
\sigma_{k}(\kappa)=\frac{1}{k!}\delta_{j_{1}\cdots j_{k}}^{i_{1}\cdots i_{k}}h_{j_{1}}^{i_{1}}\cdots h_{j_{k}}^{i_{k}}.\label{eq:matrix}
\end{equation}

By definition and (\ref{eq:matrix}), we obtain (\ref{eq:Tk1}) as
follows:
\begin{eqnarray*}
[T_{k}]_{i}^{j} & = & \frac{1}{k!}\delta_{jj_{1}\cdots j_{k}}^{ii_{1}\cdots i_{k}}h_{j_{1}}^{i_{1}}\cdots h_{j_{k}}^{i_{k}}\\
 & = & \frac{1}{k!}\delta_{j_{1}\cdots j_{k}}^{i_{1}\cdots i_{k}}h_{j_{1}}^{i_{1}}\cdots h_{j_{k}}^{i_{k}}\delta_{i}^{j}-\frac{1}{(k-1)!}\delta_{j_{1}j_{2}\cdots j_{k}}^{ii_{2}\cdots i_{k}}h_{j_{1}}^{j}h_{j_{2}}^{i_{2}}\cdots h_{j_{k}}^{i_{k}}\\
 & = & \sigma_{k}\delta_{i}^{j}-[T_{k-1}]_{i}^{k}h_{k}^{j}.
\end{eqnarray*}

For $k=1$, the symmetry of the $(2,0)$-tensor of $T_{1}$ is obviously
from the symmetry of $h$. Inductively, we assume the symmetry of
$(2,0)$-tensor $T_{k}$ is true when $k=m$. From (\ref{eq:Tk1}),
we have
\begin{eqnarray*}
[T_{m+1}]^{ij} & = & [T_{m+1}]_{l}^{i}g^{lj}=\sigma_{m+1}\delta_{l}^{i}g^{lj}-[T_{m}]_{l}^{p}h_{p}^{i}g^{lj}\\
 & = & \sigma_{m+1}g^{ij}-[T_{m}]^{pj}h_{p}^{i}.
\end{eqnarray*}

On the other hand, by (\ref{eq:Tk2}) we have
\begin{eqnarray*}
[T_{m+1}]^{ji} & = & [T_{m+1}]_{l}^{j}g^{li}=\sigma_{m+1}\delta_{l}^{j}g^{li}-[T_{m}]_{p}^{j}h_{l}^{p}g^{li}\\
 & = & \sigma_{m+1}g^{ji}-[T_{m}]_{p}^{j}h^{pi}\\
 & = & \sigma_{m+1}g^{ji}-[T_{m}]^{jp}h_{p}^{i}.
\end{eqnarray*}

So from the symmetry of $g$ and $T_{m}$, we have proved (\ref{eq:Tij}).
\end{proof}
\begin{lem}
\label{lem:areaEstimate} If $u$ satisfies the scalar equation (\ref{eq:scalar}),
then the following integral is bounded
\[
\int_{B_{r}(x_{0})}(\sigma_{1}f-\sigma_{3})dx\leq C,
\]

where $C$ depends only on $\lVert f\rVert_{C^{1}(B_{r+1}(x_{0}))}$
and $\lVert u\rVert_{C^{1}(B_{r+1}(x_{0}))}$.
\end{lem}

\begin{proof}
For a non-negative function $\phi\in C_{0}^{\infty}(B_{r+1}(x_{0}))$
with $|\nabla\phi|+|\nabla^{2}\phi|\leq C$, we assume that $\phi\equiv1$
in $B_{r}(x_{0})$ and $0\leq\phi\leq1$ in $B_{r+1}(x_{0})$. It
is obvious that the first part of the integral is bounded as follows
\begin{eqnarray*}
\int_{B_{r}(x_{0})}f\sigma_{1}dx\leq C\int_{B_{r+1}(x_{0})}\phi^{2}\sigma_{1}dx & = & C\int\phi^{2}div(\frac{Du}{W})dx\\
 & = & C\int-\sum_{i}(\phi^{2})_{i}\frac{u_{i}}{W}dx\\
 & \leq & C.
\end{eqnarray*}

Then we estimate the second part
\begin{eqnarray*}
3\int_{B_{r}(x_{0})}-\sigma_{3}dx & \leq & -3\int_{B_{r+1}(x_{0})}\phi^{2}\sigma_{3}dx\\
 & = & -\int\phi^{2}[T_{2}]_{i}^{j}D_{j}(\frac{u_{i}}{W})dx\\
 & = & 2\int\phi[T_{2}]_{i}^{j}\phi_{j}\frac{u_{i}}{W}dx.
\end{eqnarray*}

Using (\ref{eq:Tk1}), we continue our estimate
\begin{eqnarray*}
\int\phi[T_{2}]_{i}^{j}\phi_{j}\frac{u_{i}}{W}dx & = & \int\phi\phi_{i}\frac{u_{i}}{W}\sigma_{2}dx-\int\phi[T_{1}]_{i}^{k}\phi_{j}\frac{u_{i}}{W}D_{k}(\frac{u_{j}}{W})dx\\
 & \leq & C+\int[T_{1}]_{i}^{k}(\phi\phi_{j})_{k}\frac{u_{i}}{W}\frac{u_{j}}{W}dx+\int[T_{1}]_{i}^{k}D_{k}(\frac{u_{i}}{W})\phi\phi_{j}\frac{u_{j}}{W}dx\\
 & \leq & C,
\end{eqnarray*}

where we have used (\ref{eq:lap}) and the scalar curvature equation
(\ref{eq:scalar}) in the last inequality.
\end{proof}

\section{An important differential inequality}

Let us consider the quantity of $b(x):=\log\sigma_{1}$. In dimension
three, we have a very important differential inequality.
\begin{lem}
\label{lem3} For admissible solutions of the equation (\ref{eq:scalar})
in $\mathbb{R}^{3}$, we have
\begin{equation}
\sigma_{2}^{ij}b{}_{ij}\geq\frac{1}{100}\sigma_{2}^{ij}b{}_{i}b{}_{j}-C(f\sigma_{1}-\sigma_{3})+g^{ij}b_{i}d_{\nu}f(e_{j}),\label{eq:logb}
\end{equation}
where $C$ depends only on $\lVert f\rVert_{C^{2}}$, $\lVert\frac{1}{f}\rVert_{L^{\infty}}$
and $\lVert u\rVert_{C^{1}}$.
\end{lem}

\begin{rem*}
We do not know whether the corresponding higher dimensional inequalties
(\ref{eq:logb}) hold or not. This is one of the difficulty to generalize
our thereom in higher dimensions.
\end{rem*}
\begin{proof}
The calculation was done in Lemma 3 of \cite{qiu2017interior}. We
give its details in the appendix.
\end{proof}

\section{Mean value inequality.}

In this section we prove a mean value type inequality. So we can transform
the pointwise estimate into the integral estimate which is easier
to deal with. It is unclear for higher dimensional scalar curvature
equations. This is the second difficulty to generalize our theorem
in higher dimensions.
\begin{thm}
\label{thm-meanvalue} Suppose $u$ are admissible solutions of equation
(\ref{eq:scalar}) on $B_{10}\subset\mathbb{R}^{3}$, then we have
for any $y_{0}\in B_{2}$
\end{thm}

\begin{equation}
\sup_{B_{1}}b=b(y_{0})\leq C\int_{B_{1}(y_{0})}b(x)(\sigma_{1}f-\sigma_{3})dx,\label{eq:meanvalue}
\end{equation}
where $C$ depends only on $\lVert f\rVert_{C^{2}}$, $\lVert\frac{1}{f}\rVert_{L^{\infty}}$
and $\lVert u\rVert_{C^{1}}$.
\begin{proof}
Because the graph $X^{\Sigma}:=(X,\nu)=(x_{1},x_{2},x_{3},u,\frac{u_{1}}{W},\frac{u_{2}}{W},\frac{u_{3}}{W},-\frac{1}{W})$
where $u$ satisfies equation (\ref{eq:scalar}) can be viewed as
a three dimensional smooth submanifold in $(\mathbb{R}^{4}\times\mathbb{R}^{4},f(X,\nu)\sum\limits _{i=1}^{i=4}dx_{i}^{2}+\sum\limits _{i=1}^{i=4}dy_{i}^{2})$
.

When $f=1$, we shall see it is a submanifold with bounded mean curvature.
This is the key observation in the paper.

In fact, we have
\begin{eqnarray*}
X_{i}^{\Sigma} & = & (X_{i},\nu_{i})=(X_{i},h_{i}^{k}X_{k}),
\end{eqnarray*}

and
\[
G_{ij}=<X_{i}^{\Sigma},X_{j}^{\Sigma}>_{\mathbb{R}^{4}\times\mathbb{R}^{4}}=g_{ij}+h_{i}^{k}h_{kj}.
\]

If we denote $\sigma_{2}^{ij}:=\frac{\partial\sigma_{2}}{\partial h_{ij}}$,
we can verify that
\[
G^{ij}=\frac{\sigma_{2}^{ij}}{\sigma_{1}-\sigma_{3}}.
\]

Then we can prove that the mean curvature is bounded as follows:
\begin{eqnarray*}
\mathscr{\lvert H}\rvert & \leq & \lvert G{}^{ij}X_{ij}^{\Sigma}\rvert\\
 & = & \lvert\frac{\sigma_{2}^{ij}}{\sigma_{1}-\sigma_{3}}(-h_{ij}\nu,h_{ij}^{k}X_{k}-h_{i}^{k}h_{kj}\nu)\rvert\\
 & \leq & \lvert\frac{(-2\sigma_{2}\nu,-(\sigma_{1}-3\sigma_{3})\nu)}{\sigma_{1}-\sigma_{3}}\rvert\leq C.
\end{eqnarray*}

By Michael-Simon's mean value inequalities we get the estimate (\ref{eq:mean})
for scalar curvature equations.

When $f=f(X,\nu)$, we write down the details of this proof in the
appendix.
\end{proof}

\section{Proof of the theorem \ref{thm:scalar} }
\begin{proof}
From Theorem \ref{thm-meanvalue}, we have at the maximum point $x_{0}$
of $\bar{B}_{1}(0)$
\begin{eqnarray}
b(x_{0}) & \leq & \int_{B_{1}(x_{0})}b(\sigma_{1}f-\sigma_{3})dx.\label{eq:WY1}
\end{eqnarray}

We shall estimate the first part $\int_{B_{1}(x_{0})}b\sigma_{1}fdx$
in the above integral at first. Recalling that
\begin{equation}
\sigma_{2}^{ij}b{}_{ij}\geq\frac{1}{100}\sigma_{2}^{ij}b{}_{i}b{}_{j}-C(\sigma_{1}f-\sigma_{3})+g^{ij}d_{\nu}f(e_{i})b_{j},\label{eq:logb-1}
\end{equation}
we have an integral version of this inequality for any $r<5$,
\begin{eqnarray}
\int_{B_{r+1}}-\sigma_{2}^{ij}\phi_{i}b_{j}dM & \geq & c_{0}\int_{B_{r+1}}\phi\sigma_{2}^{ij}b_{i}b_{j}dM\label{eq:intlog}\\
 &  & -C[\int_{B_{r+1}}(\sigma_{1}f-\sigma_{3})\phi dx\nonumber \\
 &  & +\int_{B_{r+1}}g^{ij}d_{\nu}f(e_{i})b_{j}\phi dM].\nonumber
\end{eqnarray}
for all non-negative $\phi\in C_{0}^{\infty}(B_{r+1})$. We choose
different cutoff functions. They are all denoted by $0\le\phi\leq1$,
which support in larger ball $B_{r+1}(x_{0})$ and equals to $1$
in smaller ball $B_{r}(x_{0})$ with $|D\phi|+|D^{2}\phi|\leq C$.
\begin{eqnarray}
\int_{B_{1}(x_{0})}b\sigma_{1}dx & \leq & \int_{B_{2}(x_{0})}\phi b\sigma_{1}dx\nonumber \\
 & \leq & C(\int_{B_{2}(x_{0})}bdx+\int_{B_{2}(x_{0})}|Db|dx)\nonumber \\
 & \leq & C(1+\int_{B_{2}(x_{0})}|Db|dx).\label{eq:WY2}
\end{eqnarray}
 We only need to estimate $\int_{B_{2}(x_{0})}|Db|dx$. We use
\[
\sigma_{1}\sigma_{2}^{ij}\geq c\delta_{ij},
\]
to get
\[
\int_{B_{2}(x_{0})}|Db|dx\leq C\int_{B_{2}(x_{0})}\sqrt{\sigma_{2}^{ij}b_{i}b_{j}}\sqrt{\sigma_{1}}dx.
\]
By Holder inequality,
\begin{eqnarray}
\int_{B_{2}(x_{0})}|Db|dx & \leq & (\int_{B_{2}(x_{0})}\sigma_{2}^{ij}b_{i}b_{j}dx)^{\frac{1}{2}}(\int_{B_{2}(x_{0})}\sigma_{1}dx)^{\frac{1}{2}}\nonumber \\
 & \leq & C\int_{B_{3}(x_{0})}\phi^{2}\sigma_{2}^{ij}b_{i}b_{j}dx.\label{eq:WY3}
\end{eqnarray}
 Then using (\ref{eq:intlog}) and Lemma \ref{lem:areaEstimate},
we get
\begin{eqnarray*}
\int_{B_{3}(x_{0})}\phi^{2}\sigma_{2}^{ij}b_{i}b_{j}dx & \leq & C\int_{B_{3}(x_{0})}\phi^{2}\sigma_{2}^{ij}b_{i}b_{j}dM\\
 & \leq & C[-\int_{B_{3}(x_{0})}\phi\sigma_{2}^{ij}\phi_{i}b_{j}dM+\int_{B_{3}(x_{0})}\phi^{2}(\sigma_{1}f-\sigma_{3})dx\\
 &  & +\int_{B_{3}(x_{0})}\phi^{2}|Db|dM]\\
 & \leq & C(\int_{B_{3}(x_{0})}\sqrt{\phi^{2}\sigma_{2}^{ij}b_{i}b_{j}}\sqrt{\sigma_{2}^{kl}\phi_{k}\phi_{l}}dx+1+\int_{B_{3}(x_{0})}\phi^{2}\sqrt{\sigma_{2}^{ij}b_{i}b_{j}}\sqrt{\sigma_{1}}dx).
\end{eqnarray*}
 By Cauchy-Schwarz inequality
\begin{eqnarray*}
\int_{B_{3}(x_{0})}\phi^{2}\sigma_{2}^{ij}b_{i}b_{j}dx & \leq & C(\epsilon\int_{B_{3}(x_{0})}\phi^{2}\sigma_{2}^{ij}b_{i}b_{j}dx+\int_{B_{3}(x_{0})}\sigma_{2}^{ij}\phi_{i}\phi_{i}dx+\frac{1}{\epsilon})\\
 & \leq & C\epsilon\int_{B_{3}(x_{0})}\phi^{2}\sigma_{2}^{ij}b_{i}b_{j}dx+\frac{C}{\epsilon}.
\end{eqnarray*}
We choose $\epsilon$ small such that $C\epsilon\leq\frac{1}{2}$,
\begin{eqnarray}
\int_{B_{3}(x_{0})}\phi^{2}\sigma_{2}^{ij}b_{i}b_{j}dx & \leq & C.\label{eq:WY4}
\end{eqnarray}
So far we have obtained the estimate for the first part of (\ref{eq:WY1}),
by combining (\ref{eq:WY2}), (\ref{eq:WY3}), and (\ref{eq:WY4}).
We have
\begin{equation}
\int_{B_{1}(x_{0})}bf\sigma_{1}dx\leq C.\label{eq:part1}
\end{equation}

The second part is to estimate $\int_{B_{1}(x_{0})}-b\sigma_{3}dx$.
Thanks to the divergence free property, we integral by parts as follows
\begin{eqnarray}
-3\int_{B_{2}(x_{0})}\phi^{2}b\sigma_{3}dx & = & -\int\phi^{2}b[T_{2}]_{i}^{j}D_{j}(\frac{u_{i}}{W})dx\nonumber \\
 & = & \underbrace{\int[T_{2}]_{i}^{j}(\phi^{2})_{j}b\frac{u_{i}}{W}}_{I}dx+\underbrace{\int[T_{2}]_{i}^{j}\phi^{2}b_{j}\frac{u_{i}}{W}}_{II}dx.\label{eq:sigma3}
\end{eqnarray}

We estimate $I$ by applying (\ref{eq:Tk1}),
\begin{eqnarray}
I & = & \int(\sigma_{2}\delta_{i}^{j}-[T_{1}]_{i}^{k}h_{k}^{j})(\phi^{2})_{j}b\frac{u_{i}}{W}dx\nonumber \\
 & \leq & \int bdx-\int[T_{1}]_{i}^{k}D_{k}(\frac{u_{j}}{W})(\phi^{2})_{j}b\frac{u_{i}}{W}dx\nonumber \\
 & \leq & C+\int[T_{1}]_{i}^{k}\frac{u_{j}}{W}(\phi^{2})_{jk}b\frac{u_{i}}{W}dx+\int[T_{1}]_{i}^{k}\frac{u_{j}}{W}(\phi^{2})_{j}b_{k}\frac{u_{i}}{W}dx\nonumber \\
 &  & +2\int\sigma_{2}\frac{u_{j}}{W}(\phi^{2})_{j}bdx\nonumber \\
 & \leq & C+\int\sigma_{1}bdx+\int[T_{1}]^{kl}b_{k}g_{li}\frac{u_{i}}{W}\frac{u_{j}}{W}(\phi^{2})_{j}dx.\label{eq:I}
\end{eqnarray}

The seconde term of (\ref{eq:I}) can be estimated by the same argument
as before. We only need to estimate the last term of (\ref{eq:I}).
By Cauchy-Schwarz inequality and (\ref{eq:WY4}), we have
\begin{eqnarray}
\int[T_{1}]^{kl}b_{k}g_{li}\frac{u_{i}}{W}\sum_{j}(\frac{u_{j}}{W})(\phi^{2})_{j}dx & \leq & 4\int\phi^{2}[T_{1}]^{ij}b_{i}b_{j}dx+4\int[T_{1}]^{ij}g_{ik}\frac{u_{k}}{W}g_{jl}\frac{u_{l}}{W}(\sum_{p}\frac{u_{p}\phi{}_{p}}{W})^{2}dx\nonumber \\
 & \leq & C.\label{eq:I2}
\end{eqnarray}

From (\ref{eq:I}) and (\ref{eq:I2}) we obtain
\begin{equation}
I=\int[T_{2}]_{i}^{j}(\phi^{2})_{j}b\frac{u_{i}}{W}dx\leq C.\label{eq:Ifinal}
\end{equation}

Now we deal with $II$ by using (\ref{eq:Tk2})
\begin{eqnarray}
II & \leq & \int(\sigma_{2}\delta_{i}^{j}-[T_{1}]_{k}^{j}h_{i}^{k})\phi^{2}b_{j}\frac{u_{i}}{W}dx\nonumber \\
 & \leq & C\int\lvert Db\rvert dx-\int[T_{1}]^{jk}h_{ki}\phi^{2}b_{j}\frac{u_{i}}{W}dx.\label{eq:II}
\end{eqnarray}

As before, the first term of (\ref{eq:II}) is already estimated by
(\ref{eq:WY3}) and (\ref{eq:WY4}). We compute the second term of
(\ref{eq:II})
\begin{eqnarray*}
-\int[T_{1}]^{jk}h_{ki}\phi^{2}b_{j}\frac{u_{i}}{W}dx & \leq & 2\int\phi^{2}[T_{1}]^{ji}b_{j}b_{i}dx+2\int[T_{1}]^{ij}h_{ik}\frac{u_{k}}{W}h_{jl}\frac{u_{l}}{W}\phi^{2}dx\\
 & \leq & 2\int\phi^{2}[T_{1}]^{ji}b_{j}b_{i}dx+2\int\sigma_{2}\frac{u_{i}u_{j}}{W^{2}}h_{ij}\phi^{2}dx-2\int\sigma_{3}\frac{\lvert Du\rvert^{2}}{W^{2}}\phi^{2}dx.
\end{eqnarray*}

By (\ref{eq:WY4}) and Lemma \ref{lem:areaEstimate}, we get the estimate
for $II$,
\begin{equation}
II=\int[T_{2}]_{i}^{j}\phi^{2}b_{j}\frac{u_{i}}{W}dx\leq C.\label{eq:IIfinal}
\end{equation}

With the estimate (\ref{eq:Ifinal}) and (\ref{eq:IIfinal}) for $I$
and $II$ , we get
\begin{equation}
\int_{B_{1}(x_{0})}-b\sigma_{3}dx\leq C.\label{eq:part2}
\end{equation}

Finally, combining (\ref{eq:part1}) and (\ref{eq:part2}), we get
the estimate
\begin{equation}
\log\sigma_{1}(x_{0})\leq C.
\end{equation}
\end{proof}

\section{Appendix}

\subsection{Proof of the Lemma \ref{lem3}.}
\begin{proof}
We may choose an orthonormal frame and assume that $\{h_{ij}\}$ is
diagonal at the point. The differential equation of $b$ by using
Lemma \ref{lem1} is
\begin{eqnarray*}
A:=\sigma_{2}^{ij}b_{ij}-\epsilon\sigma_{2}^{ij}b_{i}b_{j} & \geq & \frac{\sum\limits _{i}(\sum\limits _{k\neq p}h_{kpi}^{2}-\sum\limits _{k\neq p}h_{kki}h_{ppi})}{\sigma_{1}}\\
 &  & -\frac{(1+\epsilon)\sigma_{2}^{ii}(\sum\limits _{k}h_{kki})^{2}}{\sigma_{1}^{2}}\\
 &  & -C(f\sigma_{1}-\sigma_{3})+g^{ij}b_{i}d_{\nu}f(e_{j}).
\end{eqnarray*}
We use (\ref{eq:f1}) to substitute terms with $h_{iii}$ in $A$,
\begin{eqnarray*}
A & \geq & \frac{6h_{123}^{2}}{\sigma_{1}}+\frac{2\sum_{k\neq p}h_{kpp}^{2}}{\sigma_{1}}+\sum_{k\neq p}\frac{2h_{kkp}}{\sigma_{1}}(\frac{\sum\limits _{i\neq p}\sigma_{2}^{ii}h_{iip}-f_{p}}{\sigma_{2}^{pp}})\\
 &  & -\frac{2h_{113}h_{223}+2h_{112}h_{332}+2h_{221}h_{331}}{\sigma_{1}}\\
 &  & -C(f\sigma_{1}-\sigma_{3})+g^{ij}b_{i}d_{\nu}f(e_{j})\\
 &  & -\frac{(1+\epsilon)\sigma_{2}^{ii}(\sum_{k\neq i}h_{kki}-\frac{\sum\limits _{k\neq i}\sigma_{2}^{kk}h_{kki}}{\sigma_{2}^{ii}}+\frac{f_{i}}{\sigma_{2}^{ii}})^{2}}{\sigma_{1}^{2}}.
\end{eqnarray*}

Due to symmetry, we only need to give the lower bound of the terms
which contain $h_{221}$ and $h_{331}$. We denote these terms by
$A_{1}$.
\begin{eqnarray*}
A_{1} & := & \frac{2(\sigma_{2}^{11}+\sigma_{2}^{22})h_{221}^{2}}{\sigma_{1}\sigma_{2}^{11}}+\frac{2(\sigma_{2}^{11}+\sigma_{2}^{33})h_{331}^{2}}{\sigma_{1}\sigma_{2}^{11}}-\frac{2(h_{221}+h_{331})f_{1}}{\sigma_{1}\sigma_{2}^{11}}\\
 &  & +\frac{2(\sigma_{2}^{22}+\sigma_{2}^{33}-\sigma_{2}^{11})h_{221}h_{331}}{\sigma_{1}\sigma_{2}^{11}}\\
 &  & -\frac{(1+\epsilon)[(\lambda_{2}-\lambda_{1})h_{221}+(\lambda_{3}-\lambda_{1})h_{331}+f_{1}]^{2}}{\sigma_{1}^{2}\sigma_{2}^{11}}.
\end{eqnarray*}
Then we use Cauchy-Schwarz inequality and Lemma \ref{lemma2},
\begin{eqnarray}
-\frac{(1+\epsilon)[(\lambda_{2}-\lambda_{1})h_{221}+(\lambda_{3}-\lambda_{1})h_{331}+f_{1}]^{2}}{\sigma_{1}^{2}\sigma_{2}^{11}} & \geq\nonumber \\
-\frac{(1+2\epsilon)[(\lambda_{2}-\lambda_{1})h_{221}+(\lambda_{3}-\lambda_{1})h_{331}]^{2}}{\sigma_{1}^{2}\sigma_{2}^{11}}-\frac{C}{\epsilon}\sigma_{1}.\label{eq:cauchy1}
\end{eqnarray}
Similarly, we have
\begin{eqnarray}
-\frac{2(h_{221}+h_{331})f_{1}}{\sigma_{1}\sigma_{2}^{11}} & \geq & -\frac{2\epsilon^{2}\sigma_{1}(h_{221}+h_{331})^{2}}{\sigma_{1}\sigma_{2}^{11}}-\frac{f_{1}^{2}}{2\epsilon^{2}\sigma_{2}^{11}\sigma_{1}^{2}}\nonumber \\
 & \geq & -\frac{2\epsilon^{2}\sigma_{1}(h_{221}+h_{331})^{2}}{\sigma_{1}\sigma_{2}^{11}}-\frac{C}{\epsilon^{2}}\sigma_{1}.\label{eq:cauchy2}
\end{eqnarray}
 Then we substitute (\ref{eq:cauchy1}) and (\ref{eq:cauchy2}) into
$A_{1}$ to get
\begin{eqnarray*}
A_{1} & \geq & \frac{2\sigma_{2}^{11}+2\sigma_{2}^{22}}{\sigma_{1}\sigma_{2}^{11}}h_{221}^{2}+\frac{2\sigma_{2}^{11}+2\sigma_{2}^{33}}{\sigma_{1}\sigma_{2}^{11}}h_{331}^{2}\\
 &  & +\frac{4\lambda_{1}}{\sigma_{1}\sigma_{2}^{11}}h_{221}h_{331}-\frac{2\epsilon^{2}\sigma_{1}(h_{221}+h_{331})^{2}}{\sigma_{1}\sigma_{2}^{11}}\\
 &  & -\frac{(1+2\epsilon)[(\lambda_{2}-\lambda_{1})h_{221}+(\lambda_{3}-\lambda_{1})h_{331}]^{2}}{\sigma_{1}^{2}\sigma_{2}^{11}}\\
 &  & -\frac{C}{\epsilon^{2}}\sigma_{1}.
\end{eqnarray*}
 We will show the Claim \ref{claim1} and Claim \ref{claim2} in the
below. Then we choose $\epsilon=\frac{1}{100}$, such that
\[
1+\delta\geq\frac{1+2\epsilon}{1-\epsilon},
\]
where $\delta$ is small constant in the Claim \ref{claim2}.

In all, we shall get
\begin{eqnarray*}
A & = & \sum_{i=1}^{3}A_{i}-C(f\sigma_{1}-\sigma_{3})+g^{ij}b_{i}d_{\nu}f(e_{j})\\
 & \geq & -C(f\sigma_{1}-\sigma_{3})+g^{ij}d_{\nu}f(e_{i})b_{j}.
\end{eqnarray*}
\end{proof}
\begin{claim}
\label{claim1} For any $\epsilon\leq\frac{2}{3}$, we have
\[
\frac{2\sigma_{2}^{11}+2\sigma_{2}^{22}}{\sigma_{1}\sigma_{2}^{11}}h_{221}^{2}+\frac{2\sigma_{2}^{11}+2\sigma_{2}^{33}}{\sigma_{1}\sigma_{2}^{11}}h_{331}^{2}+\frac{4\lambda_{1}}{\sigma_{1}\sigma_{2}^{11}}h_{221}h_{331}\geq\frac{2\epsilon\sigma_{1}(h_{221}+h_{331})^{2}}{\sigma_{1}\sigma_{2}^{11}}.
\]
\end{claim}

\begin{proof}
This claim follows from the following elementary inequality
\begin{eqnarray*}
 & (\sigma_{2}^{11}+\sigma_{2}^{22}-\epsilon\sigma_{1})(\sigma_{2}^{11}+\sigma_{2}^{33}-\epsilon\sigma_{1})-(\lambda_{1}-\epsilon\sigma_{1})^{2}\\
= & (1-\epsilon)^{2}\sigma_{1}^{2}+(1-\epsilon)\sigma_{1}(\lambda_{2}+\lambda_{3})+\lambda_{2}\lambda_{3}-(\lambda_{1}-\epsilon\sigma_{1})^{2}\\
= & 3(1-\epsilon)f+(2-3\epsilon)(\lambda_{2}^{2}+\lambda_{2}\lambda_{3}+\lambda_{3}^{2}).
\end{eqnarray*}

If we assume $2-3\epsilon\geq0$, we have above inequality nonnegative.
\end{proof}
\begin{claim}
\label{claim2} For any $\delta\leq\frac{1}{20}$ , we have
\begin{eqnarray*}
\frac{2\sigma_{2}^{11}+2\sigma_{2}^{22}}{\sigma_{1}\sigma_{2}^{11}}h_{221}^{2}+\frac{2\sigma_{2}^{11}+2\sigma_{2}^{33}}{\sigma_{1}\sigma_{2}^{11}}h_{331}^{2}+\frac{4\lambda_{1}}{\sigma_{1}\sigma_{2}^{11}}h_{221}h_{331}\\
\geq\frac{(1+\delta)[(\lambda_{2}-\lambda_{1})h_{221}+(\lambda_{3}-\lambda_{1})h_{331}]^{2}}{\sigma_{1}^{2}\sigma_{2}^{11}}.
\end{eqnarray*}
\end{claim}

\begin{proof}
We can compute the coefficient in front of $\frac{h_{221}^{2}}{\sigma_{1}^{2}\sigma_{2}^{11}}$
\begin{eqnarray*}
2(\sigma_{2}^{11}+\sigma_{2}^{22})\sigma_{1}-(1+\delta)(\lambda_{1}-\lambda_{2})^{2} & = & (1-\delta)\lambda_{1}^{2}+(1-\delta)\lambda_{2}^{2}+4\lambda_{3}^{2}+6f+2\delta\lambda_{1}\lambda_{2}\\
 & = & (1-\delta)(\lambda_{1}+\frac{\delta}{1-\delta}\lambda_{2})^{2}+\frac{1-2\delta}{1-\delta}\lambda_{2}^{2}+4\lambda_{3}^{2}+6f.
\end{eqnarray*}

And similarly, the coefficient in front of $\frac{h_{331}^{2}}{\sigma_{1}^{2}\sigma_{2}^{11}}$
is
\[
(1-\delta)(\lambda_{1}+\frac{\delta}{1-\delta}\lambda_{3})^{2}+\frac{1-2\delta}{1-\delta}\lambda_{3}^{2}+4\lambda_{2}^{2}+6f.
\]

We also compute the coefficient of $\frac{2h_{221}h_{331}}{\sigma_{1}^{2}\sigma_{2}^{11}}$
\begin{eqnarray*}
2\lambda_{1}\sigma_{1}-(1+\delta)(\lambda_{1}-\lambda_{2})(\lambda_{1}-\lambda_{3}) & = & (1-\delta)\lambda_{1}^{2}+(3+\delta)f-2(2+\delta)\lambda_{2}\lambda_{3}\\
 & = & (1-\delta)(\lambda_{1}+\frac{\delta}{1-\delta}\lambda_{2})(\lambda_{1}+\frac{\delta}{1-\delta}\lambda_{3})\\
 &  & -\frac{4-3\delta}{1-\delta}\lambda_{2}\lambda_{3}+3f.
\end{eqnarray*}

It is easy to see that for any small $\delta$
\begin{eqnarray*}
[\frac{1-2\delta}{1-\delta}\lambda_{2}^{2}+4\lambda_{3}^{2}][\frac{1-2\delta}{1-\delta}\lambda_{3}^{2}+4\lambda_{2}^{2}] & \geq & [-\frac{4-3\delta}{1-\delta}\lambda_{2}\lambda_{3}]^{2}
\end{eqnarray*}

and
\begin{eqnarray*}
(6f)^{2} & \geq & (3f)^{2}.
\end{eqnarray*}

We have proved that the coefficient matrix in front of $h_{221}^{2}$,
$h_{331}^{2}$ and $2h_{221}h_{331}$ is positive definite. So we
complete the proof of this claim.
\end{proof}

\subsection{Proof of the theorem \ref{thm-meanvalue}.}
\begin{proof}
We prove this theorem similar to Michael-Simon \cite{michael1973sobolev}.
First from Lemma \ref{lem3}, we have
\begin{eqnarray*}
\sigma_{2}^{ij}b{}_{ij} & \geq & -C(f\sigma_{1}-\sigma_{3})+g^{ij}b_{i}d_{\nu}f(e_{j}).
\end{eqnarray*}
 Let $\chi$ be a non-negative, non-decreasing function in $C^{1}(\mathbb{R})$
with support in the interval $(0,\infty)$ and set

\[
\psi(r)=\int_{r}^{\infty}t\chi(\rho-t)dt
\]

where $0<\rho<10,$ and $r^{2}=f(X(x),\nu(x))|X(x)-X(y_{0})|{}^{2}+2-2(\nu(x),\nu(y_{0}))$.

Let us denote
\[
\mathfrak{B}_{\rho}=\{x\in B_{10}(y_{0}):\quad f(X(x),\nu(x))|X(x)-X(y_{0})|{}^{2}+2-2(\nu(x),\nu(y_{0}))\leq\rho^{2}\}.
\]
 We may assume that $(X(y_{0}),\nu(y_{0}))=(0,E_{4})$.

By direct computation, we have
\begin{equation}
2rr_{i}=f_{i}|X|^{2}+2f(X,e_{i})-2h_{ki}(e_{k},E_{4}),\label{eq:r1}
\end{equation}

and
\begin{eqnarray}
2r_{i}r_{j}+2rr_{ij} & = & f_{ij}|X|^{2}+2f_{i}(X,e_{j})+2f_{j}(X,e_{i})+2f\delta_{ij}\nonumber \\
 &  & -2fh_{ij}(X,\nu)-2h_{kij}(e_{k},E_{4})+2h_{ki}h_{kj}(\nu,E_{4}).\label{eq:r2}
\end{eqnarray}

Because $\lambda_{2}+\lambda_{3}>0$, we may assume $\lambda_{3}<0$,
and $\lambda_{1}\geq\lambda_{2}\geq\lambda_{3}$,
\begin{eqnarray}
f\sigma_{1}-\sigma_{3} & \geq & f\lambda_{1}+\lambda_{1}\lambda_{3}^{2}\nonumber \\
 & \geq & -c\lambda_{1}\lambda_{3}\nonumber \\
 & \geq & -c\lambda_{2}\lambda_{3}.\label{eq:lambda1}
\end{eqnarray}

By equation we also have
\begin{equation}
\lambda_{2}\lambda_{1}=f-\lambda_{1}\lambda_{3}-\lambda_{2}\lambda_{3}\leq C(f\sigma_{1}-\sigma_{3}).\label{eq:lambda2}
\end{equation}

We then have from (\ref{eq:r1}), (\ref{eq:r2}), (\ref{eq:lambda1}),
(\ref{eq:lambda2}) and Lemma \ref{lem1}
\begin{eqnarray}
\sigma_{2}^{ij}\psi_{ij} & = & \sigma_{2}^{ij}(-r_{i}r\chi(\rho-r))_{j}\nonumber \\
 & = & -\sigma_{2}^{ij}r_{ij}r\chi(\rho-r)-\sigma_{2}^{ij}r_{i}r_{j}\chi(\rho-r)+\sigma_{2}^{ij}r_{i}r_{j}r\chi^{\prime}(\rho-r)\nonumber \\
 & = & -\chi(\rho-r)\sigma_{2}^{ij}[f\delta_{ij}-fh_{ij}(X,\nu)+\frac{f_{ij}|X|^{2}}{2}+2f_{i}(X,e_{j})]\nonumber \\
 &  & +\chi(\rho-r)\sigma_{2}^{ij}[h_{kij}(e_{k},E_{4})-h_{ki}h_{kj}(\nu,E_{4})]\nonumber \\
 &  & +\sigma_{2}^{ij}r_{i}r_{j}r\chi^{\prime}(\rho-r)\nonumber \\
 & \leq & -3(\sigma_{1}f-\sigma_{3})\chi+C(r^{2}\chi+r\chi)(f\sigma_{1}-\sigma_{3})\nonumber \\
 &  & +\sigma_{2}^{ij}r_{i}r_{j}r\chi^{\prime}.\label{eq:psi}
\end{eqnarray}

We claim that
\begin{equation}
\sigma_{2}^{ij}r_{i}r_{j}\leq(f\sigma_{1}-\sigma_{3})(1+Cr).\label{eq:cla}
\end{equation}

In fact,
\begin{eqnarray*}
\frac{\sigma_{2}^{ii}[f_{i}|X|^{2}+2f(X,e_{i})-2\sum_{k}h_{ki}(e_{k},E_{4})]^{2}}{4r^{2}}\\
\leq C(f\sigma_{1}-\sigma_{3})r+\frac{\sigma_{2}^{ii}[f(X,e_{i})-h_{ii}(e_{i},E_{4})]^{2}}{r^{2}}.
\end{eqnarray*}

Moreover, we have following elementary properties
\begin{equation}
(f\sigma_{1}-\sigma_{3})\delta_{ij}-f\sigma_{2}^{ij}=\sigma_{2}^{kl}h_{ki}h_{lj}.
\end{equation}

and
\begin{eqnarray*}
f\sigma_{2}^{ii}h_{ii}^{2}(X,e_{i})^{2}+2f\sigma_{2}^{ii}(X,e_{i})(e_{i},E_{4})h_{ii}+f\sigma_{2}^{ii}(e_{i},E_{4})^{2} & \geq\\
f\sigma_{2}^{ii}[h_{ii}(X,e_{i})+(e_{i},E_{4})]^{2}.
\end{eqnarray*}

Then we obtain (\ref{eq:cla}) by
\[
\frac{\sigma_{2}^{ii}[f(X,e_{i})-h_{ii}(e_{i},E_{4})]^{2}}{r^{2}}\leq(f\sigma_{1}-\sigma_{3}).
\]

We obtain from (\ref{eq:cla}) and (\ref{eq:psi}) that
\[
\sigma_{2}^{ij}\psi_{ij}\leq(\sigma_{1}f-\sigma_{3})[-3\chi+C(r^{2}\chi+r\chi)+(1+Cr)r\chi^{\prime}].
\]

Then we mutiply both side by $b$ and take integral on the domain
$\mathfrak{B}_{10}$,
\begin{eqnarray}
\int_{\mathfrak{B}_{10}}b\sigma_{2}^{ij}\psi_{ij}dM & \leq & \rho^{4}\frac{d}{d\rho}(\int_{\mathfrak{B}_{10}}\frac{b\chi(\rho-r)}{\rho^{3}}(\sigma_{1}f-\sigma_{3})dM)\nonumber \\
 &  & +C\int_{\mathfrak{B}_{10}}br^{2}\chi^{\prime}(\sigma_{1}f-\sigma_{3})dM\nonumber \\
 &  & +C\int_{\mathfrak{B}_{10}}rb\chi(\rho-r)(\sigma_{1}f-\sigma_{3})dM.\label{eq:mean}
\end{eqnarray}

By (\ref{eq:logb}), we have
\begin{equation}
-C\int_{\mathfrak{B}_{10}}(\sigma_{1}f-\sigma_{3})\psi dM+\int_{\mathfrak{B}_{10}}g^{ij}d_{\nu}f(e_{i})b{}_{j}\mbox{\ensuremath{\psi}}dM\leq\int_{\mathfrak{B}_{10}}b\sigma_{2}^{ij}\psi_{ij}dM.\label{eq:step3}
\end{equation}

Inserting (\ref{eq:step3}) into (\ref{eq:mean}), we get
\begin{eqnarray*}
 & -\frac{d}{d\rho}(\int_{\mathfrak{B}_{10}}\frac{b\chi(\rho-r)}{\rho^{3}}(\sigma_{1}f-\sigma_{3})dM)\\
\leq & \frac{C\int_{\mathfrak{B}_{10}}rb\chi(\rho-r)(\sigma_{1}f-\sigma_{3})dM}{\rho^{4}}\\
 & +\frac{C\int_{\mathfrak{B}_{10}}br^{2}\chi^{\prime}(\sigma_{1}f-\sigma_{3})dM}{\rho^{4}}\\
 & +\frac{C\int_{\mathfrak{B}_{10}}(\sigma_{1}f-\sigma_{3})\psi dM}{\rho^{4}}\\
 & -\frac{\int_{\mathfrak{B}_{10}}g^{ij}d_{\nu}f(e_{i})b{}_{j}\mbox{\ensuremath{\psi}}dM}{\rho^{4}}.
\end{eqnarray*}

Because $\chi$, $\chi^{\prime}$ and $\psi$ are all supported in
$\mathfrak{B}_{\rho}$, we deal with right hand side of above inequality
term by term. For the first term, we have
\begin{equation}
\frac{C\int_{\mathfrak{B}_{10}}rb\chi(\rho-r)(\sigma_{1}f-\sigma_{3})dM}{\rho^{4}}\leq C\frac{\int_{\mathfrak{B}_{10}}b\chi(\rho-r)(\sigma_{1}f-\sigma_{3})dM}{\rho^{3}}.\label{eq:term1}
\end{equation}

Then for the second term, we integrate from $\delta$ to $R$,
\begin{eqnarray}
 & \int_{\delta}^{R}\frac{\int_{\mathfrak{B}_{10}}br^{2}\chi^{\prime}(\sigma_{1}f-\sigma_{3})dM}{\rho^{4}}d\rho\nonumber \\
\leq & \int_{\delta}^{R}\frac{\int_{\mathfrak{B}_{10}}b\chi^{\prime}(\sigma_{1}f-\sigma_{3})dM}{\rho^{2}}d\rho\nonumber \\
\leq & \frac{\int_{\mathfrak{B}_{10}}b\chi(\sigma_{1}f-\sigma_{3})dM}{\rho^{2}}|_{\delta}^{R}\nonumber \\
 & +\int_{\delta}^{R}\frac{2\int_{\mathfrak{B}_{10}}b\chi(\sigma_{1}f-\sigma_{3})dM}{\rho^{3}}d\rho.\label{eq:term2}
\end{eqnarray}

For the third term, we use the definition of $\psi$ to estimate
\begin{eqnarray}
 & \frac{C\int_{\mathfrak{B}_{10}}b(\sigma_{1}f-\sigma_{3})\psi dM}{\rho^{4}}\nonumber \\
\leq & \frac{C\int_{\mathfrak{B}_{10}}b\chi(\rho-r)(\sigma_{1}f-\sigma_{3})dM}{\rho^{3}}.\label{eq:term3}
\end{eqnarray}

For the last term, we use (\ref{eq:r1}) and the definition of $\psi$
to get
\begin{eqnarray}
 & -\frac{\int_{\mathfrak{B}_{10}}g^{ij}d_{\nu}f(e_{i})b_{j}\psi dM}{\rho^{4}}\nonumber \\
\leq & \frac{C\int_{\mathfrak{B}_{10}}b\sigma_{1}\psi dM-\int_{\mathfrak{B}_{10}}g^{ij}d_{\nu}f(e_{i})r_{j}br\chi dM}{\rho^{4}}\nonumber \\
\leq & \frac{C[\int_{\mathfrak{B}_{10}}b\sigma_{1}\psi dM+\int_{\mathfrak{B}_{10}}(\sigma_{1}f-\sigma_{3})br\chi dM]}{\rho^{4}}\nonumber \\
\leq & \frac{C\int_{\mathfrak{B}_{10}}b\chi(\rho-r)(\sigma_{1}f-\sigma_{3})dM}{\rho^{3}}.\label{eq:term4}
\end{eqnarray}

Combining (\ref{eq:term1}), (\ref{eq:term2}), (\ref{eq:term3})
and (\ref{eq:term4}), we integrate from $0\leq\delta$ to $R\leq10$,
then use Grönwall's inequality to get
\[
\int_{\mathfrak{B}_{10}}\frac{b(\sigma_{1}f-\sigma_{3})\chi(\delta-r)}{\delta^{3}}dM\leq C\int_{\mathfrak{B}_{10}}\frac{b(\sigma_{1}f-\sigma_{3})\chi(R-r)}{R^{3}}dM.
\]

Letting $\chi$ approximate the characteristic function of the interval
$(0,\infty)$, in an appropriate fashion, we obtain,
\begin{equation}
\frac{\int_{\mathfrak{B}_{\delta}}b(\sigma_{1}f-\sigma_{3})dM}{\delta^{3}}\leq C\frac{\int_{\mathfrak{B}_{R}}b(\sigma_{1}f-\sigma_{3})dM}{R^{3}}.\label{eq:monotonicity}
\end{equation}

Because the graph $(X,\nu)$ where $u$ satisfied equation (\ref{eq:sigma2})
can be viewed as a three dimensional smooth submanifold in $(\mathbb{R}^{4}\times\mathbb{R}^{4},f(\sum\limits _{i=1}^{4}dx_{i}^{2})+\sum\limits _{i=1}^{4}dy_{i}^{2})$
with volume form exactly $(\sigma_{1}f-\sigma_{3})dM$. Moreover,
for a sufficient small $\delta>0$, the geodesic ball with radius
$\delta$ of this submanifold is comparable with $\mathfrak{B}_{\delta}$.
Then let $\delta\rightarrow0$, we finally get
\begin{eqnarray*}
b(y_{0}) & \leq & C\frac{\int_{\mathfrak{B}_{R}(\bar{y}_{0})}b(\sigma_{1}f-\sigma_{3})dM}{R^{3}}\leq C\frac{\int_{B_{R}(y_{0})}b(\sigma_{1}f-\sigma_{3})dx}{R^{3}}.
\end{eqnarray*}
\end{proof}
\begin{acknowledgement*}
The author would like to express gratitude to Professor Pengfei Guan for supports and many helpful discussions when he did postdoctoral research at McGill University.
\end{acknowledgement*}
\bibliographystyle{plain}

\end{document}